\documentclass[11pt,twoside, reqno]{amsart}
\usepackage{amsfonts}
\usepackage{latexsym}
\usepackage{amssymb}
\usepackage{amsmath}
\usepackage{dsfont}
\usepackage{multicol} 
\usepackage[pdftex,svgnames]{xcolor}
\usepackage{url}
\usepackage[utf8]{inputenc}
\usepackage{mathtools}

\renewcommand{\leq}{\leqslant}

\renewcommand{\geq}{\geqslant}

\newcommand{\N}{{\mathbb N}}
\newcommand{\R}{{\mathbb R}}
\newcommand{\Ff}{{\mathbb F}}
\newcommand{\F}{\mathfrak{F}}

\newcommand{\Oo}{\mathcal{O}}

\newcommand{\B}{\mathfrak{B}}

\newcommand{\TP}{\text{TP}}
\newcommand{\Int}{\text{Int}}

\theoremstyle{plain}
\newtheorem{theorem}{Theorem}
\newtheorem{proposition}{Proposition}

\newtheorem{lemma}[theorem]{Lemma}

\theoremstyle{example}

\theoremstyle{remark}
\newtheorem{remark}[theorem]{Remark}
\theoremstyle{definition}
\newtheorem{definition}[theorem]{Definition}
\numberwithin{equation}{section}

\newcommand{\TFAE}{the following assertions are equivalent: }

\begin{document}
	\title[On integration with respect to filter]{On integration with respect to filter}
	\author{Dmytro Seliutin}
	\address{School of Mathematics and Informatics V.N. Karazin Kharkiv  National University,  61022 Kharkiv, Ukraine}
	\email{seliutin2020pg@student.karazin.ua}
	
	\begin{abstract}
		Using a concept of filter we propose one generalization of Riemann integral, that is integration with respect to filter. We study this problem, demonstrate different properties and phenomena of filter integration.
	\end{abstract}
	
	\maketitle
	\section{Introduction}
	Let us remind main concepts which we use in this paper. Throughout this article $\Omega$ stand for a non-empty set. Non-empty family of subsets $\F \subset 2^\Omega$ is called \textit{filter on $\Omega$}, if $\F$ satisfies the following axioms:
	\begin{enumerate}
		\item $\emptyset \notin \F$;
		\item if $A,\ B \in \F$ then $A \cap B \in \F$;
		\item if $A \in \F$ and $D \supset A$ then $D \in \F$.
	\end{enumerate}
	
	Also very useful for us is a concept of filter base. Non-empty family of subsets $\B \subset 2^\Omega$ is called \textit{filter base on $\Omega$}, if $\emptyset \notin \B$ and for every $A,\ B \in \B$ there exists $C \in \B$ such that $C \subset A \cap B$. We say that filter base \textit{generates filter $\F$} if and only if for each $A \in \F$ there is $B \in \B$ such that $B \subset A$.
	
	Let $X$ be a topological vector space, $f: X \rightarrow \R$ be a function. For $t \in X$ denote $\Oo(t)$ the family of all neighbourhoods of $t$. Let $\F$ be a filter on $X$, $y \in \R$. Function $f$ is said to be \textit{convergent to $y$ over filter $\F$} (denote $y = \lim\limits_{\F} f$), if for each $U \in \Oo(y)$ there exists $A \in \F$ such that for each $t \in A$ the following holds true: $f(t) \in U$. We refers, for example, to \cite{kadets} for more information about filter and related concepts. 
	
	The concept of filter is a very powerful tool for studying different properties of general topological vector spaces. For example, in \cite{(2)listan-garcia} authors study convergence over ideal, generated by the modular function. Ideal is a concept dual to filter. In \cite{kad_sel_comp} we study completeness and its generalization using filters. 
	
	In this this article we refer our attention to classical Riemann integral. Let us remind how we can construct this object. Let $[a,b] \subset \R$, let $f: [a,b] \rightarrow \R$ be a continuous function. Denote $\Pi = \{a \leq \xi_1 \leq \xi_2 \leq ...\leq \xi_n = b\}$ the partition of $[a,b]$, in other words, $\overset{n}{\underset{k=1}{\cup}} [\xi_{k-1}, \xi_k] = [a,b]$. Consider also the set $T = \{t_1, t_2, ...,t_n\}$ such that for each $k=1,2,...,n$ $t_k \in [\xi_{k-1}, \xi_k]$. Let us call the pair $(\Pi, T)$ by the \textit{tagged partition on the segment}. Denote $d(\Pi)$ \textit{the diameter of the $\Pi$} --  maximum length of $[\xi_{k-1}, \xi_k]$, where $k=1,2,...,n$. Let us recall that function $f$ is said to be \textit{Riemann integrable} if there exist the limit $I = \lim\limits_{d(\Pi) \rightarrow 0} \sum\limits_{k=1}^{n} f(t_k) \cdot |\xi_k - \xi_{k-1}|$, and we call this limit the Riemann integral of the function $f$, and write $I = \int\limits_{a}^{b} f(t)dt$. We know many different properties of this integral, for example linearity, integration on subsegment of $[a,b]$ etc. 
	
	If we look at the definition of Riemann integral more attentively, we realize that, in fact we can use one special filter and obtain desirable result. In next section we are going to develop this idea.
	
	\section{Integration with respect to filter}
	
	Just for simplicity we are going to consider functions, defined on $[0,1]$. Let $f:[0,1] \rightarrow \R$ be a function. As above, denote $\Pi = \{a \leq \xi_1 \leq \xi_2 \leq ...\leq \xi_n = b\}$ the partition of $[0,1]$, in other words, $\overset{n}{\underset{k=1}{\cup}} [\xi_{k-1}, \xi_k] = [0,1]$. Consider also the set $T = \{t_1, t_2, ...,t_n\}$ such that for each $k=1,2,...,n$ $t_k \in [\xi_{k-1}, \xi_k]$. For $k=1,2,...,n$ denote $\Delta_k := |\xi_k - \xi_{k-1}|$. Denote also $\TP[0,1]$ the set of all tagged partition of $[0,1]$. For a tagged partition $(\Pi, T) \in \TP[0,1]$ denote 
	
	$$
	S(f, \Pi, T) = \sum_{k=1}^{n} f(t_k) \Delta_k.
	$$
	Now we are going to introduce the central definition of this paper
	
	\begin{definition}\label{main_def}
		Let $f:[0,1] \rightarrow \R$ be a function, $\F$ be a filter on $\TP[0,1]$. We say that $f$ is \textit{integrable over filter $\F$} ($\F$-integrable for short), if there exists $I \in \R$ such that $I = \lim\limits_{\F} S(f, \Pi, T)$. The number $I$ is called \textit{the $\F$-integral of the $f$} (denote $I = \int\limits_{0}^{1} f d \F$).
	\end{definition}

	\begin{remark}
		The fact that $f$ is $\F$-integrable we will write as follows: 
		$$
		f \in \Int(\F).
		$$
	\end{remark}
	
	\begin{remark}
		Using Definition \ref{main_def} we can construct the Riemann integral as follows. Let $\delta > 0$ be a real positive number. Denote 
		$$
		P_{<\delta} = \{(\Pi, T) \in \TP[0,1]:\ d(\Pi) < \delta\},
		$$
		where $d(\Pi)$ stands for diameter of $\Pi$. Consider now
		$$
		\B_{<\delta} = \{P_{<\delta}: \delta > 0\}.
		$$
		It is easy to check that $\B_{<\delta}$ is a filter base. Denote $\F_{<\delta}$ filter generate by $\B_{<\delta}$. Let $f:[0,1] \rightarrow \R$ be a function. Then $f$ is integrable by Riemann if there exists the limit $\lim\limits_{\F_{<\delta}} S(f, \Pi, T)$.
	\end{remark}
	
	Bellow we study  different properties of filter integration. 
	\begin{definition}
		Let $X$ be a non-empty set, $f: X \rightarrow \R$ be a function, and $\F$ be a filter on $X$. We say that $f$ is \textit{bounded with respect to $\F$} ($\F$-bounded for short), if there is $C > 0$ such that there exists $A \in \F$ such that for every $t \in A$ $|f(t)| < C$.
	\end{definition}
	
	The following lemma is very simple, but for readers convenient we present its proof.
	
	\begin{lemma}\label{lemma_bounded}
		Let $X$ be a non-empty set, $f: X \rightarrow \R$ be a function, and $\F$ be a filter on $X$. Suppose that there exists $I \in \R$, $I = \lim\limits_{\F} f$. Then $f$ is $\F$-bounded.
	\end{lemma}
	
	\begin{proof}
		We know that $I = \lim\limits_{\F} f$. It means that for every $\varepsilon > 0$ there exists $A \in \F$ such that for all $t \in A$ $|f(t) - I| < \varepsilon$. Consider
		$$
		|f(t)| - |I|\leq|f(t) - I| < \varepsilon.
		$$
		In other words, $|f(t)| \leq |I| + \varepsilon$. Then just put $C := |I| + \varepsilon$.
	\end{proof}

	The next theorem generalizes well-know fact about Riemann integral: if function in integrable by Riemann then it's bounded.
	
	\begin{theorem}
		Let $\F$ be a filter on $\TP[0,1]$, $f:[0,1] \rightarrow \R$ be a function, and $f \in \Int(\F)$. Then $S(f, \Pi, T)$ is $\F$-bounded.
	\end{theorem}
	
	\begin{proof}
		Just use Lemma \ref{lemma_bounded}.
	\end{proof}
	
	Let us formulate well-known fact about Riemann integral, using filters. 
	\begin{theorem}\label{bound}
		Let $f: [0,1] \rightarrow \R$, there exists $\lim\limits_{\F_{<\delta}} S(f, \Pi, T)$. Then $f$ is bounded, in other words, there is $C > 0$ such that for all $t \in [0,1]$ $|f(t)| \leq C$.
	\end{theorem}
	
	The next theorem is natural generalization of the Theorem \ref{bound}.
	
	\begin{theorem}
		Let $f: [0,1] \rightarrow \R$, let $\F$ be a filter on $TP[0,1]$ such that for every $A \in \F$ there exists $B \in \F_{<\delta}$ such that $B \subset A$ and let there exists $I \in \R$ such that $I = \lim\limits_{\F} S(f,\Pi, T)$. Then $C > 0$ such that for each $t \in [0,1]$ we have $|f(t)| < C$.
	\end{theorem}
	
	\begin{proof}
		There exists $I \in \R$ such that $I = \lim\limits_{\F} S(f,\Pi, T)$ $\Leftrightarrow$ for all $\varepsilon > 0$ there exists $A \in \F$ such that for all $(\Pi, T) \in A$ $|S(f, \Pi, T) - I| < \varepsilon$. We know that for $A \in \F$ there is $B \in \F_{<\delta}$ such that $B \subset A$, then, particularly, for all $\varepsilon > 0$ there exists $A \in \F$ there is $B \in \F_{<\delta}$ such that $B \subset A$ such that for all $(\Pi, T) \in B$ $|S(f, \Pi, T) - I| < \varepsilon$ $\Rightarrow$ for all $\varepsilon > 0$ there exists $B \in \F_{<\delta}$ such that for all $(\Pi, T) \in B$ $|S(f, \Pi, T) - I| < \varepsilon$ $\overset{\text{Theorem \ref{bound}}}{\Rightarrow}$ $C > 0$ such that for each $t \in [0,1]$ we have $|f(t)| < C$, in other words, $f$ is bounded.
	\end{proof}

	Now we are going to demonstrate that filter integration has additive property. To demonstrate this we proof next easy two lemmas.
	
	\begin{lemma}\label{sum}
		Let $X$ be a non-empty set, $f,\ g: X \rightarrow \R$ be a functions, and $\F$ be a filter on $X$. Let $x = \lim\limits_{\F} f$, $y = \lim\limits_{\F} g$. Then $\lim\limits_{\F} (f+g) = x+y$.
	\end{lemma}
	
	\begin{proof}
		We know that $x = \lim\limits_{\F} f$, so for each $U \in \Oo(x)$ there is $A \in \F$ such that $f(A) \subset U$. Analogically, $y = \lim\limits_{\F} f$, it means that for each $V \in \Oo(x)$ there is $B \in \F$ such that $f(B) \subset V$. We have to demonstrate that for each $W \in \Oo(x + y)$ there exists $C \in \F$ such that $(f+g)(C) \subset W$. Let fix $W \in \Oo(x+y)$. Then there exist $W_1 \in \Oo(x)$ and $W_2 \in \Oo(y)$ such that $W \supset W_1 + W_2$. Then there are $C_1,\ C_2 \in \F$ such that $f(C_1) \subset W_1$ and $f(C_2) \subset W_2$. Denote $C := C_1 \cap C_2$. Clearly that $C \in \F$. So
		$$
		(f+g)(C) = f(C) + g(C) \subset W_1 + W_2 \subset W.
		$$
	\end{proof}
	
	\begin{lemma}\label{mult}
		Let $X$ be a non-empty set, $f: X \rightarrow \R$ be a function, $\F$ be a filter on $X$, and $\alpha \in \R$. Let $x = \lim\limits_{\F} f$. Then $\lim\limits_{\F} \alpha f = \alpha x$.
	\end{lemma}

	\begin{proof}
	$x = \lim\limits_{\F} f$, it means that for each $U \in \Oo(x)$ there is $A \in \F$ such that $f(A) \subset U$. We have to demonstrate that for all $V \in \Oo(\alpha x)$ there is $B \in \F$ such that $(\alpha f)(B) \subset V$. Suppose that $\alpha \neq 0$. The case $\alpha = 0$ is obvious. Remark that if $W \in \Oo(x)$ then $\alpha W \in \Oo(\alpha x)$. So just put $B := A$. Then $(\alpha f)(B) = \alpha f(B) \subset \alpha U \in \Oo(\alpha x)$.
	\end{proof}

	\begin{theorem}\label{add_mult}
		Let $\F$ be a filter on $\TP[0,1]$, $f,g:[0,1] \rightarrow \R$ be a functions, $f \in \Int(\F)$, $\alpha,\ \beta \in \R$, and $f \in \Int(\F)$ and $g \in \Int(\F)$. Then $(\alpha f + \beta g) \in \Int(\F)$
	\end{theorem}

	\begin{proof}
		Just use Lemmas \ref{sum} and \ref{mult}.
	\end{proof}
	
	\section{Integration with respect to different filters}
	
	In the previous section we've studied arithmetic properties of integral over filter and problems deals with boundedness. This section is devoted to integration over different filters and its relations.
	
	For $(\Pi, T) \in \TP[0,1]$ and $t \in T$ we denote $\Delta(t)$ length of the element of partition of $\Pi$ which covers $t$.
	
	Let $(\Pi_1, T_1), (\Pi_2, T_2)$ be partitions of $[0,1]$. Consider
	
	\begin{equation}\label{dist}
		\begin{split}
			\rho((\Pi_1, T_1), (\Pi_2, T_2)) = 	\\
			\sum_{t \in T_1 \cap T_2} |\Delta_1(t) - \Delta_2(t)| + \sum_{T_1\setminus T_2} \Delta_1(t) + \sum_{T_2\setminus T_1} \Delta_2(t).
		\end{split}
	\end{equation}
	
	For easy using of concept defined in Equation \ref{dist} consider $\Ff: [0,1] \rightarrow l_1[0,1]$, such that $\Ff(t) = e_t$, where
	$$
	e_t(\tau) = 
	\begin{cases}
		1, \text{if } \tau = t;\\
		0, \text{otherwise}.
	\end{cases}
	$$
	
	It is clearly then that
	$$
	\rho((\Pi_1, T_1), (\Pi_2, T_2)) = ||S(\Ff, \Pi_1, T_1) - S(\Ff, \Pi_2, T_2)||.
	$$
	
	Now we are going to demonstrate that the mapping $\rho$, defined above, is a metric, or distance between two tagged partitions.
	
	\begin{proposition}
		Consider $\rho: \TP[0,1] \times \TP[0,1] \rightarrow \R$, $\rho((\Pi_1, T_1), (\Pi_2, T_2)) = ||S(\Ff, \Pi_1, T_1) - S(\Ff, \Pi_2, T_2)||$. Then $\rho$ satisfies all metric axioms.
	\end{proposition}
	
	\begin{proof}
			\begin{enumerate}
			\item let $(\Pi_1, T_1) = (\Pi_2, T_2)$. 
			
			It is clear that in this case $\rho((\Pi_1, T_1), (\Pi_2, T_2)) = 0$;
			\item let $\rho((\Pi_1, T_1), (\Pi_2, T_2)) = 0$. 
			
			Then $\rho((\Pi_1, T_1), (\Pi_2, T_2)) = \sum\limits_{t \in T_1 \cap T_2} |\Delta_1(t) - \Delta_2(t)| + \sum\limits_{T_1\setminus T_2} \Delta_1(t) + \sum\limits_{T_2\setminus T_1} \Delta_2(t) = 0$. We have a sum of non-negative numbers equals to 0. This means that 
			\begin{itemize}
				\item $\forall t \in T_1 \cap T_2$ $|\Delta_1(t) - \Delta_2(t)| = 0$ $\Rightarrow$ $\forall t \in T_1 \cap T_2$ $\Delta_1(t) = \Delta_2(t)$;
				\item $\forall t \in T_1 \setminus T_2$ $ \Delta_1(t)= 0$;
				\item $\forall t \in T_2 \setminus T_1$ $ \Delta_2(t)= 0$;
			\end{itemize}
			$\Rightarrow$ $(\Pi_1, T_1) = (\Pi_2, T_2)$.
			\item consider $(\Pi_1, T_1), (\Pi_2, T_2), (\Pi_3, T_3)$. Then 
			\begin{equation*}
				\begin{split}
					\rho((\Pi_1, T_1), (\Pi_2, T_2)) =\\
					||S(\Ff, \Pi_1, T_1) - S(\Ff, \Pi_2, T_2) + S(\Ff, \Pi_3, T_3) - S(\Ff, \Pi_3, T_3)|| \leq\\
					||S(\Ff, \Pi_1, T_1) - S(\Ff, \Pi_3, T_3)|| + ||S(\Ff, \Pi_3, T_3) - S(\Ff, \Pi_2, T_2)|| =\\ \rho((\Pi_1, T_1), (\Pi_3, T_3)) + \rho((\Pi_3, T_3), (\Pi_2, T_2))
				\end{split}
			\end{equation*}
		\end{enumerate}
	\end{proof}
	
	Now we introduce very important concept. 
	
	\begin{definition}
		Let $\F_1, \F_2$ be filters on $TP[0,1]$. We say that $\F_2$ \textit{$\rho$-dominates} filter $\F_1$ ($\F_2 \succ_{\rho} \F_1$), if for every $\varepsilon < 0$ and for each $A_1 \in \F_1$ there exists $A_2 \in \F_2$ such that for all $(\Pi_2, T_2) \in A_2$ there is $(\Pi_1, T_1) \in A_1$ such that $\rho((\Pi_1, T_1), (\Pi_2, T_2)) < \varepsilon$.
	\end{definition}
	
	\begin{proposition}
		Let $\F_2 \supset \F_1$. Then $\F_2$ $\rho$-dominates $\F_1$.
	\end{proposition}
	
	\begin{proof}
		As $\F_2 \supset \F_1$ we obtain that if $A \in \F_1$ then $A \in \F_2$. Consider an arbitrary $\varepsilon > 0$. Then for every $A_1 \in \F_1$ there is $A_2 \in \F_2$, $A_2 := A_1$ such that for each $(\Pi_2, T_2) \in A_2$ there exists $(\Pi_1, T_1) \in A_1$, $(\Pi_1, T_1) := (\Pi_2, T_2)$ such that $\rho\left((\Pi_1, T_1), (\Pi_2, T_2)\right) = \rho\left((\Pi_2, T_2), (\Pi_2, T_2)\right) = 0 < \varepsilon$.
	\end{proof}
	
	Previous proposition shows us that $\rho$-dominance generates some relation of order on $\TP[0,1]$ and is more general concept that relation of inclusion.
	
	It is clear that if $\F_1 \subset \F_2$ and $f \in \Int(\F_1)$ then $f \in \Int(\F_2)$ -- just use the definition of function limit over filter. So we can formulate next easy proposition.
	
	\begin{proposition}
		Let $f: [0,1] \rightarrow \R$ be a function, $\F_1,\ \F_2$ be filters on $\TP[0,1]$ such that $\F_1\ \subset \F_2$ and $f \in \Int(\F_1)$. Then $f \in \Int(\F_2)$.
	\end{proposition}
	
		\begin{theorem}
		Let $\F_1, \F_2$ be filters on $[0,1]$. Let $f: [0,1] \rightarrow \R$ be a bounded function. Let $I = \lim\limits_{\F_1} S(f, \Pi, T)$ and $\F_2 \succ_{\rho} \F_1$. Then $I = \lim\limits_{\F_2} S(f, \Pi, T)$.
	\end{theorem}
	
	\begin{proof}
		Denote $C := \sup\limits_{t \in [0,1]} |f(t)|$. 
		
		We have to proof that for every $\varepsilon > 0$ there exists $B \in \F_2$ such that for each $(\Pi_B, T_B) \in B$ we have $|S(f, \Pi_B, T_B) - I| < \varepsilon$. 
		
		We know that for every $\varepsilon > 0$ there exists $A \in \F_1$ such that for each $(\Pi_1, T_1) \in A$ we have $|S(f, \Pi_1, T_1) - I| < \varepsilon$. 
		
		Now for an arbitrary $\varepsilon > 0$ and $A \in \F_1$ found above one can find $A_2 \in \F_2$ such that for all $(\Pi_2, T_2) \in A_2$ there is $(\Pi_1, T_1) \in A_1$ such that $\rho((\Pi_1, T_1), (\Pi_2, T_2)) < \varepsilon$. Then put $B := A_2$. Then for all $(\Pi_B, T_B) \in B$ we have $(\Pi_1, T_1) \in A_1$ such that
		\begin{align*}
			|S(f, \Pi_B, T_B) - I| =\\
			|S(f, \Pi_B, T_B) - S(f, \Pi_1, T_1) + S(f, \Pi_1, T_1) -I| \leq\\
			|S(f, \Pi_B, T_B) - S(f, \Pi_1, T_1)| + |S(f, \Pi_1, T_1) -I| = \\
			\sum_{t \in T_B\cap T_1} |f(t)|\cdot|\Delta_B(t) - \Delta_1(t)| + \sum_{t \in T_B \setminus T_1} |f(t)|\cdot\Delta_B(t) +\\ \sum_{t \in T_1 \setminus T_B} |f(t)|\cdot\Delta_1(t) + \varepsilon \leq C\cdot \rho((\Pi_B, T_B), (\Pi_1, T)) + \varepsilon \leq\\ C\varepsilon + \varepsilon \leq \varepsilon(1 + C).
		\end{align*} 
	\end{proof}

	\section{Exactly tagged filters}
	
	In this part of our paper we consider problems deals filter integration of unbounded functions.
	
	\begin{definition}
		Let $\B$ be a filter base on $TP[0,1]$. We say that $\B$ is \textit{exactly tagged} if there exist $A\subset [0,1]$ -- a strictly decreasing sequence of numbers such that for each $B \in \B$ and for every $(\Pi, T) \in B$ we have that $T \cap A = \emptyset$.
	\end{definition}
	
	\begin{definition}
		We say that filter $\F$ on $TP[0,1]$ is \textit{exactly tagged} if there exists exactly tagged base $\B$ of $\F$.
	\end{definition}
	
	\begin{theorem}
		If filter $\F$ on $TP[0,1]$ is exactly tagged then there exists unbounded function $f: [0,1] \rightarrow \R$ such that $f \in Int(\F)$.
	\end{theorem}
	
		\begin{proof}
		Denote $\displaystyle \N^{-1} =  \left\{ \frac{1}{n} \right\}_{n \in \N}$ and consider next filter base $\B = (B_n)_{n \in \N}$ on $TP[0,1]$:\\
		$\displaystyle B_1 = \left\{(\Pi, T): T \cap \N^{-1} = \emptyset \text{ and } d(\Pi) < 1\right\}$;\\
		$\displaystyle B_2 = \left\{(\Pi, T): T \cap \N^{-1} = \emptyset \text{ and } d(\Pi) < \frac{1}{2}\right\}$;\\
		$\displaystyle B_3 = \left\{(\Pi, T): T \cap \N^{-1} = \emptyset \text{ and } d(\Pi) < \frac{1}{3}\right\}$;\\
		...\\
		$\displaystyle B_m = \left\{(\Pi, T): T \cap \N^{-1} = \emptyset \text{ and } d(\Pi) < \frac{1}{m}\right\}$. 
		
		Consider now
		
		$$
		f(t) = 
		\begin{cases}
			\displaystyle
			n, \text{ if } t = \frac{1}{n},\ n \in \N \\
			0, \text{ otherwise}
		\end{cases}.
		$$
		
		Then for each $n \in \N$ and for every $(\Pi, T) \in B_n$ we have that $S(f, \Pi, T) = 0$, so $\lim\limits_{\B}S(f, \Pi, T) = 0$.
	\end{proof}
	
	For a tagged partition $(\Pi, T)$ of $[0, 1]$ and $\tau \in [0, 1]$ denote $\ell(\Pi, T, \tau)$ the number which is equal to the length of the segment $\Delta \in \Pi$, for which $\tau \in \Delta$, if $\tau \in T$. If $\tau \notin T$, we put $\ell(\Pi, T, \tau) = 0$. In this notation
	$$
	S(f, \Pi, T) = \sum_{t \in [0,1]} f(t) \ell(\Pi, T, t).
	$$
	
	\begin{theorem} \label{th}
		For a filter $\F$ on $TP[0,1]$ \TFAE
		\begin{enumerate}
			\item There exists an unbounded function $f:[0,1] \rightarrow [0, +\infty)$ such that $S(f, \Pi, T)$ is $\F$-bounded;
			\item There exists a countable subset $\{t_n\}_{n \in \N} \subset [0, 1]$ such that there is $A \in \F$ such that for every  $(\Pi, T) \in A$
			$$
			\sum_{n \in \N} n \cdot \ell(\Pi, T, t_n) < 1.
			$$
		\end{enumerate}
	\end{theorem}
	
		\begin{proof}
		\textbf{{(1)$\Rightarrow$(2)}}: Let $f$ be a non-negative, unbounded function on $[0,1]$ such that there is $C > 0$ and $B \in \F$ such that for each $(\Pi, T) \in B$ we have $\sum\limits_{t \in [0,1]} f(t) \cdot \ell(\Pi, T, t) < C$. As $f$ is unbounded, there exists $(\alpha_n) \subset [0,1]$ such that for every $n \in \N$ $f(\alpha_n) \geq C n$. Then there exists $(\alpha_n) \subset [0,1]$, $C > 0$, there is $A \in \F$, $A := B$ such that for all $(\Pi, T) \in A$ we obtain:
		\begin{align*}
			\sum_{t \in [0,1]} n \cdot \ell(\Pi, T, \alpha_n) \leq \sum_{n \in \N} \frac{f(\alpha_n)}{C} \cdot \ell(\Pi, T, \alpha_n) \leq\\
			\frac{1}{C}\sum_{t \in [0,1]} f(t) \cdot \ell(\Pi, T, t) < \frac{1}{C} \cdot C = 1.
		\end{align*}
		
		\textbf{{(2)$\Rightarrow$(1)}}: Let there exists a countable subset $\{t_n\}_{n \in \N} \subset [0, 1]$ and $C > 0$ such that there is $A \in \F$ such that for every  $(\Pi, T) \in A$ $\sum\limits_{n \in \N} n \cdot \ell(\Pi, T, t_n) < C$. Consider function
		$$
		f(t) = 
		\begin{cases}
			n, \text{ if } t = \alpha_n,\ n \in \N\\
			0, \text{ if } t \neq \alpha_n
		\end{cases}.
		$$
		
		Obviously, $f(t)$ is unbounded. Then there is $C > 0$ and there is $B \in \F$, $B := A$ such that for every $(\Pi, T) \in A$
		
		\begin{align*}
			\sum_{t \in [0,1]} f(t) \cdot \ell(\Pi, T, t) \leq \sum_{n \in \N} f(\alpha_n) \cdot \ell(\Pi, T, \alpha_n) \leq\\
			\sum_{n \in \N} n \cdot \ell(\Pi, T, \alpha_n) < C
		\end{align*}.
	\end{proof}
	
	\section{Integration over filter on a subsegment}
	
	Our next goal is as follows: if function $f$ is integrable on $[0,1]$ over filter $\F$ on $\TP[0,1]$ then for an arbitrary $[\alpha, \beta] \subset [0,1]$ function $f$ is is integrable on $[\alpha, \beta]$ over filter $\F$. 
	
	To achieve this purpose we need to construct some restriction of filter $\F$ on subsegment $[\alpha, \beta] \subset [0,1]$. Now we present how we can construct such restriction.
	
	Consider an arbitrary $[\alpha, \beta] \subset [0,1]$. We consider only $T$ such that $T \cap (\alpha, \beta) \neq \emptyset$. Consider an arbitrary $(\Pi, T) \in TP[0,1]$.
	
	We have four cases:
	
	\begin{enumerate}
		\item $\min \{T \cap (\alpha, \beta)\} > \min\{\Pi \cap (\alpha, \beta)\}$\\
		$\max \{T \cap (\alpha ,\beta)\} < \max \{\Pi \cap (\alpha ,\beta)\}$;
		
		\item $\min \{T \cap (\alpha, \beta)\} > \max \{\Pi \cap (0, \alpha)\}$\\
		$\max \{T \cap (\alpha ,\beta)\}  < \max \{\Pi \cap (\alpha ,\beta)\}$;
		
		\item  $\min \{T \cap (\alpha, \beta)\}  > \min\{\Pi \cap (\alpha, \beta)\}$\\
		$\max \{T \cap (\alpha ,\beta)\} < \min\{\Pi \cap (\beta, 1)\}$;
		
		\item $\min \{T \cap (\alpha, \beta)\} > \max \{\Pi \cap (0, \alpha)$\\
		$\max \{T \cap (\alpha ,\beta)\} < \min\{\Pi \cap (\beta, 1)\}$. 
		
	\end{enumerate}
	
	We have to construct a restriction of $(\Pi, T)$ on $[\alpha, \beta]$. In each of four described cases we have such $(\Pi_k, T_k) \in TP[\alpha, \beta]$, $k =1, 2, 3, 4$:
	
	\begin{enumerate}
		\item $\Pi_1 = \bigg(\Pi \setminus \big((\Pi \cap [0, \alpha)) \cup (\Pi \cap (\beta, 1)) \cup \min\{\Pi \cap (\alpha, \beta)\} \cup \max \{\Pi \cap (\alpha ,\beta)\}\big)\bigg) \cup \{\alpha, \beta\}$\\
		$T_1 = T \setminus \big((T \cap [0, \alpha)) \cup (T \cap (\beta, 1])\big)$;
		
		\item $\Pi_2 = \bigg(\Pi \setminus \big((\Pi \cap [0, \alpha)) \cup \max\{\Pi \cap (\alpha, \beta)\} \cup (\Pi \cap (\beta, 1))\big)\bigg)\cup \{\alpha, \beta\}$\\
		$T_2 = T_1$;
		
		\item $\Pi_3 = \bigg(\Pi \setminus \big((\Pi \cap [0,\alpha)) \cup \min\{\Pi \cap (\alpha)\} \cup (\Pi \cap [\beta, 1)) \big)\bigg)\cup \{\alpha, \beta\}$\\
		$T_3 = T_1$;
		
		\item $\Pi_4 = \bigg(\Pi \setminus \big((\Pi \cap [0, \alpha)) \cup (\Pi \cap [\beta, 1))\big)\bigg)\cup \{\alpha, \beta\}$\\
		$T_4 = T_1$.
	\end{enumerate}
	
	Now if we have an arbitrary filter $\F$ on $TP[0,1]$ we can construct filter $\F_{[\alpha, \beta]}$ on $TP[\alpha, \beta]$, induced with $\F$ in such way: consider an arbitrary $A \in \F$ and for each $(\Pi, T) \in A$ we have to execute an algorithm, described above. For each $A \in \F$ denote $A_\alpha^\beta$ the restriction of $A$ on $[\alpha, \beta]$, described above.
	
	\begin{definition}
		Let $\F$ be a filter on $TP[0,1]$, $[\alpha, \beta] \subset [0,1]$. We call the filter $\F$ \textit{$[\alpha, \beta]$-complemented} if for each $A \in \F$, for every $(\Pi_1, T_1),\ (\Pi_2, T_2) \in A_\alpha^\beta$ there exists $(\Pi^*, T^*) \in TP[0, \alpha]$ and $(\Pi^{**}, T^{**}) \in TP[\beta, 1]$ such that 
		$$
		(\Pi^*, T^*) \cup (\Pi_1, T_1) \cup (\Pi^{**}, T^{**}) \in A,
		$$
		$$
		(\Pi^*, T^*) \cup (\Pi_2, T_2) \cup (\Pi^{**}, T^{**}) \in A.
		$$
	\end{definition}
	
	Here we present promised result about filter integration on subsegment.
		
	\begin{theorem}
		Let $f: [0,1] \rightarrow \R$, $\F$ be a filter on $TP[0,1]$ such that for each $[\alpha, \beta] \subset [0,1]$ $\F$ is $[\alpha, \beta]$-complemented. Let $f$ is integrated of $[0,1]$ with respect to $\F$. Then for every $[\alpha, \beta] \subset [0,1]$ $f$ is integrated on $[\alpha, \beta]$ with respect to $\F$
	\end{theorem}
	
	\begin{proof}
		We know that for an arbitrary $\varepsilon > 0$ there exists $A \in \F$ such that for all $(\Pi_1, T_1), (\Pi_2, T_2) \in A$ we have: $|S(f, \Pi_1, T_1) - S(f, \Pi_2, T_2)| < \varepsilon$. 
		
		Let fix $\varepsilon > 0$ and consider an arbitrary $[\alpha, \beta] \subset [0,1]$. For  $A \in \F$ consider an arbitrary $(\Pi^1, T^1), (\Pi^2, T^2) \in A_\alpha^\beta$. As $\F$ is $[\alpha, \beta]$-complemented we can find $(\Pi^*, T^*) \in A_0^\alpha$ and $(\Pi^{**}, T^{**}) \in A_\beta^1$ such that $(\Pi_{11}, T_{11}) := (\Pi^*, T^*) \cup (\Pi^1, T^1) \cup (\Pi^{**}, T^{**}) \in A$ and $(\Pi_{22}, T_{22}) :=(\Pi^*, T^*) \cup (\Pi^2, T^2) \cup (\Pi^{**}, T^{**}) \in A$. Then $\varepsilon > |S(f, \Pi_{11}, T_{11}) - S(f, \Pi_{22}, T_{22})| = |S(f, \Pi^1, T^1) - S(f, \Pi^2, T^2)|$.
		
	\end{proof}
	
	\vspace{8mm}
	
	\textbf{Acknowledgment.} This paper is partially supported by a grant from Akhiezer Foundation (Kharkiv). The author is thankful to his parents for their support and his scientific adviser, professor Vladimir Kadets for his constant help with this project. Also author thanks the Defense Forces of Ukraine for the defence and fight against Russian aggressors.

\end{document}